\documentclass{amsart}
\usepackage{amsmath,amssymb,amsthm}

\newtheorem{prop}{Proposition}[section]
\newtheorem{theorem}[prop]{Theorem}
\newtheorem{lemma}[prop]{Lemma}

\theoremstyle{definition}
\newtheorem{definition}[prop]{Definition}
\theoremstyle{remark}
\newtheorem{remark}[prop]{Remark}
\newtheorem{example}[prop]{Example}

\begin{document}

\title[Refocusing of Light Rays in Space-Time]{Refocusing of Light Rays in Space-Time} 
\author[P.~Kinlaw]{Paul A. Kinlaw}
\address{P.~Kinlaw, Mathematics Department, 6188 Kemeny,  Dartmouth College, Hanover NH 03755, 
USA} 
\email{Paul.Kinlaw@dartmouth.edu}

\maketitle

\begin{abstract}

We investigate refocusing and strong refocusing of light rays in a space-time.  A strongly refocusing space-time is refocusing.  The converse is unknown.  We construct examples of space-times which are refocusing, but not strongly so, at a particular point.  These space-times are strongly refocusing at other points.  The geometrization conjecture proved by Perelman implies that a globally hyperbolic refocusing space-time of dimension 
$\leq 4$ admits a strongly refocusing Lorentz metric.  

We show that the possibly empty set of points at which a strongly causal space-time is refocusing is closed.  We prove that a Lorentz covering space of a strongly causal refocusing space-time is a strongly causal refocusing space-time. This generalizes the result of Chernov and Rudyak for globally hyperbolic space-times.
\end{abstract}

\section{Conventions}
Throughout this paper all manifolds and maps are assumed to be smooth unless stated otherwise.  
\emph{Smooth} means $C^\infty$.  Manifolds are denoted by capital letters and sometimes accompanied by a superscript which indicates the dimension of the manifold.  Points are denoted by lower case letters.  Vectors are generally written in boldface for emphasis.

\section{List of the Main Results}

We establish several results on refocusing.  Their statement requires some knowledge of Lorentz geometry, which we provide in the appendix for the convenience of the reader.  We then discuss several definitions of refocusing and strong refocusing of null-geodesics in a space-time, and prove the following results.\\
\\
\textbf{(i.)}~Theorem~\ref{closedtheorem}.  Let $(X,g)$ be a strongly causal space-time.  Then the possibly empty $\{x\in X: (X,g)$ is refocusing at $x\}$ is a closed subset of $X$.\\
\\
\textbf{(ii.)}~Theorem~\ref{coveringtheorem}.  Let $(X,g)$ be a strongly causal space-time which is refocusing at $x\in X$.  Let $\rho: (\widetilde{X},\widetilde{g})\to (X,g)$ a Lorentz covering map.  Then $(\widetilde{X},\widetilde{g})$ is a strongly causal space-time which is refocusing at each point $\widetilde{x}\in \rho^{-1}(x)$.\\
\\
\textbf{(iii.)}~Example~\ref{ghexample} and~Proposition~\ref{nulltimelikecompleteprop}.  We construct a globally hyperbolic space-time which is refocusing at $x$ but not strongly refocusing at $x$, for all points $x$ of a Cauchy surface.  A conformal change gives a space-time with all of these same properties which is also timelike and null complete.\\
\\
\textbf{(iv.)}~Proposition~\ref{maintheorem4}.  Any globally hyperbolic refocusing spacetime $(X^{n+1},g), n\leq 3$ admits a strongly refocusing metric.  This follows from the Geometrization Conjecture proved by Perelman.\\
\\
\textbf{(v.)}~Proposition~\ref{maintheorem1} and Theorem~\ref{maintheorem2}. We prove that the three definitions of refocusing due to Low are in fact equivalent.  We also give a thorough proof of Low's statement that a strongly causal non-refocusing space-time is homeomorphic with its sky space.

\section{Refocusing}

In Minkowski space-time $(\mathbb{R}^{n+1},dx_1^2+...+dx_n^2-dt^2)$ the light rays are straight lines as in Euclidean space.  Therefore the light rays which pass through a given event never intersect at another event.  On the other hand, all of the light rays passing through any event $(p,t)$ in the Einstein Cylinder $(\mathbb{S}^{n}\times \mathbb{R},g=m\oplus-dt^2)$ also pass through $(-p,t+\pi)$.  (Here, $m$ denotes the standard Riemannian metric on the unit sphere $\mathbb{S}^n\subset \mathbb{R}^{n+1}$.)  This motivates the following definition.

\begin{definition}
A strongly causal space-time $(X^{n+1},g)$ is called \emph{strongly refocusing} at an event $x\in X$ if there is an event $y\neq x$ such that all null-geodesics through $y$ pass through $x$.  We call $(X,g)$ \emph{strongly refocusing} if it is strongly refocusing at some event.
\end{definition}

We will need the following proposition.

\begin{prop}\label{skyembedding}
Let $(X^{n+1},g)$ be a strongly causal space-time.  For each event $x\in X$, its sky $O_x$ is a topologically embedded $(n-1)$-sphere in $\mathcal{N}$.
\end{prop}

\begin{proof}
Define the \emph{future null bundle} $FNX\subset TX$ to be the subspace of $TX$ formed by all future-directed null vectors.  The \emph{spherical future null bundle} is defined as the quotient  $SNX=FNX/\mathbb{R}^+$ of $FNX$ by the group action of $\mathbb{R}^+$ by scalar multiplication, see Appendix C.

Let $\pi: SNX\to \mathcal{N}: \mathbf{v}\mapsto [\mathbf{v}]$ be the projection.  Let $SN_xX\subset SNX$ be the subspace of $SNX$ formed by the future null directions at $x$.  Put $F=\pi|_{SN_xX}: SN_xX\to O_x\subset \mathcal{N}$.  Then $F$ is a bijection since it clearly gives a one-to-one correspondence between future null directions at $x$ and null-geodesics through $x$.  Also $F$ is continuous since it is the restriction of the projection map.  It remains to show that $F$ is an open map.  Let $U$ be an open subset of $SN_xX$.  Then $U=SN_xX\cap V$ where $V$ is an open subset of $SNX$.  We have $F(U)=F(SN_xX\cap V)=\pi(SN_xX\cap V)=\pi(SN_xX)\cap \pi(V)$=$O_x\cap \pi(V)$.  This is an open subset of $O_x$ because $\pi(V)$ is an open subset of $\mathcal{N}$.  This is true because $\pi$ is an open map since it is a submersion with respect to the smooth structure on $\mathcal{N}$.  (See for instance \cite[p. 205]{BrickellClark} where $\mathcal{N}$ is viewed as the set of leaves of the spherical null distribution on $X$.)
\end{proof}

\begin{prop}\label{mainprop3}
Let $(X,g)$ be a strongly causal space-time.  Define a relation $\sim$ on $X$ by $x\sim y$ if and only if all null-geodesics through $x$ pass through $y$.  Then  $\sim$  is an equivalence relation, making $X/\sim$ into a quotient space with respect to the manifold topology of $X$.
\end{prop}

\begin{proof}
The relation is clearly reflexive.  It is also transitive.  (If all null-geodesics through $x$ pass through $y$ and all of the null-geodesics through $y$ pass through $z$, then all of the null-geodesics through $x$ must pass through $z$.)  We prove symmetry.

Suppose that all of the null-geodesics through $x$ pass through $y$.  Then we have the following containment of skies: $O_x\subset O_y$.  Suppose by way of contradiction that the containment $O_x\subset O_y$ is proper, i.e. there is a null-geodesic $\gamma$ belonging to $O_y$ but not $O_x$.  The skies $O_x$ and $O_y$ are topologically embedded $(n-1)$-spheres in the space $\mathcal{N}$ of null-geodesics in $X$ (modulo orientation-preserving affine reparametrization) by Proposition~\ref{skyembedding}.   Therefore the $(n-1)$-sphere $O_x$ is embedded in a proper subset of the punctured $(n-1)$-sphere $O_y\setminus \{\gamma\}$, which is homeomorphic with $\mathbb{R}^{n-1}$.  The Borsuk-Ulam Theorem states that any continuous map $f: \mathbb{S}^{n-1}\to \mathbb{R}^{n-1}$ satisfies $f(p)=f(-p)$ for some $p\in \mathbb{S}^{n-1}$, see \cite[p. 176]{Hatcher}.  Therefore, such a map fails to be injective and hence cannot be an embedding.
\end{proof}

\begin{example}
If a globally hyperbolic space-time $X$ is not strongly refocusing, then the equivalence classes are singleton sets of events and $X/ \sim = X$.  On the other hand, the quotient $\mathbb{S}^n\times [0,\pi]/\sim$ of the strongly refocusing Einstein Cylinder $\mathbb{S}^n\times \mathbb{R}$ under the equivalence relation $(p,0)\sim (-p,\pi)$ is a space-time called \emph{compactified Minkowski space-time}.  This space-time is central to the study of twistor theory, see \cite[Section 9.2, page 298]{PR}.  When $n=1$ it is homeomorphic to the torus.  
\end{example}

We begin the discussion of refocusing with a preliminary proposition regarding the local structure of strongly causal space-times.  We will use this result to demonstrate a proof of the equivalence of three definitions of refocusing introduced by Low.

\begin{prop}\label{mainprop1}
If a space-time is strongly causal at an event $p$, then for every open neighborhood $U$ of $p$ there exists an open neighborhood $V$ of $p$ contained in $U$ which is causally convex and convex normal with compact closure.  Furthermore any such $V$ is globally hyperbolic, cf~\cite[Theorem 2.14]{MS}.
\end{prop}

\begin{remark}
We say that there exist \emph{arbitrarily small} open neighborhoods of $p$ with a given property $P$ if for all open neighborhoods $U$ of $p$ there exists an open neighborhood $V$ of $p$ contained in $U$ with property $P$.

Similarly we say that all \emph{sufficiently small} open neighborhoods of $p$ have property $P$ if there exists an open neighborhood $U$ of $p$ such that all open neighborhoods $V$ of $p$ contained in $U$ have property $P$.
\end{remark}

\begin{proof}
Let $(X^{n+1},g)$ be a space-time which is strongly causal at $p\in X$ and let $U$ be an open neighborhood of $x$.  The set of events at which a space-time is strongly causal is open in $X$ \cite[Proposition 4.13]{Penrose}.  Therefore, there is an open neighborhood $W$ of $p$ contained in $U$ such that $X$ is strongly causal at all points of $W$.  (If $W$ is not contained in $U$ then replace $W$ with $W\cap U$.)  There exists a convex normal and causally convex open neighborhood $V$ of $p$ contained in $W$ \cite[Theorem 4.27]{BEE}.  By local compactness of manifolds, we may also assume $V$ has compact closure.  We show $(V,g|_{V})$ is globally hyperbolic.  The space-time $(V,g|_{V})$ is strongly causal because $(X,g)$ is strongly causal at each point of $V$ and hence $(V,g|_{V})$ is strongly causal at each point of $V$.  Let $x,y\in V$.  Then $C=J^{+}(x)\cap J^{-}(y)$ is contained in $V$ by causal convexity. By Proposition~\ref{3.4BEE} \cite[Prop 3.4]{BEE}, $C$ is a closed subset of V.

Since $X$ is strongly causal it suffices by \cite[Lemma 4.29]{BEE} to show that $cl(C)$ is compact, which is clear since $cl(C)\subset cl(V)$.
\end{proof}

We defined strong refocusing above.  We now address refocusing.  We show that the following three definitions of refocusing introduced by Low in his papers \cite[p. 5]{LowCones}, \cite[p. 6]{LowNullgeodesics}, \cite[p. 46]{LowRefocusing} are equivalent.  Low did not discuss the relationship between these definitions in his papers.  He encouraged me to write a formal proof and I greatly appreciate the conversations with him regarding this subject.

\begin{prop}\label{maintheorem1}
Let $(X,g)$ be a strongly causal space-time.  The following three definitions of \emph{refocusing} at an event $x\in X$ are equivalent.\\
 \textbf{(i.)}There exists an open neighborhood $V$ of $x$ such that for all open $U$ with $x\in U\subset V$, there exists an event $y\notin U$ such that all null-geodesics through $y$ enter $U$.\\
\textbf{(ii.)}There exists an open neighborhood $V$ of $x$ such that for all open $U$ with $x\in U\subset V$, there exists an event $y\notin V$ such that all null-geodesics through $y$ enter $U$.\\
\textbf{(iii.)}There exists an open neighborhood $V$ of $x$ and a sequence $\{y_{n}\}_{n=1}^\infty$ in $V^{c}$ such that for all open $U$ with $x\in U\subset V$, there exists $K\in \mathbb{N}$ such that if $n\geq K$ then all null-geodesics through $y_{n}$ enter $U$.
\end{prop}

\begin{definition}\label{refocusing}
A strongly causal space-time $(X,g)$ is \emph{refocusing} at an event $x\in X$ if one (and hence all) of i-iii hold above.  A strongly causal space-time is \emph{refocusing} if it is refocusing at some event.
\end{definition}

\begin{remark}
The open neighborhood $V$ of $x$ is necessary in the definition since the statement is meant to be true only for sufficiently small neighborhoods $U$ of $x$.  Otherwise, putting $U=X$ would show that all space-times would be non-refocusing for trivial reasons.

The distinction between the first two definitions of refocusing at $x$ is that "$y\notin U$" is replaced by "$y\notin V$".

Note that the definition of strong refocusing and all three definitions of refocusing make sense even when $(X,g)$ is not strongly causal.  However in this case we do not know if these three definitions remain equivalent.  In his works Low considered refocusing only for strongly causal space-times.  The implications (iii.)$\Rightarrow$(ii.)$\Rightarrow$(i.) hold for arbitrary space-times but our proof of the implication (i.)$\Rightarrow$(iii.) uses strong causality.

Clearly, strong refocusing implies refocusing.  The converse is not clear.  Refocusing is sometimes referred to as \emph{weak refocusing} to emphasize the distinction from strong refocusing.  The word refocusing is also spelled \emph{refocussing}.
\end{remark}

\begin{proof}
The implication (iii)$\Rightarrow$(ii) is obvious since one puts $y=y_{K}$.  The implication (ii)$\Rightarrow$(i) is also obvious since $V^{c}\subset U^{c}$.

We show (i)$\Rightarrow$(iii).  Suppose (i.) and note that we can replace $V$ with any open neighborhood of $x$ contained in $V$.  The space-time $X$ is strongly causal so by Proposition~\ref{mainprop1} we may choose $V$ sufficiently small so that it is causally convex and convex normal with compact closure and thus a globally hyperbolic open neighborhood of $x$.

Let $\{U_{n}\}_{n=1}^\infty$ be a countable decreasing neighborhood base at $x$ contained in $V$.  (Such a neighborhood base exists for topological reasons).  Moreover, we can choose this base to be of the form $U_{n}=I^{+}(p_{n})\cap I^{-}(q_{n})$ since by Theorem~\ref{KP}, \cite[Theorem 4.24]{Penrose} the Alexandrov and manifold topologies agree for a strongly causal space-time.  This yields a sequence $\{y_{n}\}_{n=1}^\infty$ of events in $X$ such that for each $n\in \mathbb{N}$ we have $y_{n}\notin U_{n}$ and all null-geodesics through $y_n$ enter $U_n$.  Without loss of generality we may assume $U_n\subset V$ for all $n\in \mathbb{N}$.  It suffices also to assume that all but finitely many $y_n$ lie inside $V$, since if infinitely many $y_n$ lie outside $V$ we are done.

Since $cl(V)$ is compact, there exists a subsequence $\{y_{n_k}\}_{k=1}^\infty$ in $V$ converging to a point $\hat{y}\in cl(V)$.  If $\hat{y}\neq x$, then there exist disjoint open neighborhoods $W_1, W_2$ of $\hat{y},x$, respectively.  Assume without loss of generality that $V\subset W_{2}$.  Since $y_{n_k}$ converges to $\hat{y} \in W_1$, we have $y_{n_k}\in V$ for $k$ sufficiently large, giving (iii).  It remains to show that $\hat{y}\neq x$.

By way of contradiction, suppose $y_{n_k}$ converges to $x$.  Re-indexing gives a sequence $\{y_n\}_{n=1}^\infty$ in $V$ converging to $x$ such that all light rays through $y_n$ enter $U_n$, where $\{U_n\}_{n=1}^\infty$ is a countable decreasing base at $x$ contained in $V$ with $U_n=I^{+}(p_n)\cap I^-(q_n)$ and $y_n\notin U_n$.  We show this is impossible.  For convenience, given $n\in \mathbb{N}$ put $y=y_n$, $U=U_n= I^+(p)\cap I^-(q), p=p_n, q=q_n$.  We will show that such a point $y$ does not exist and hence such a sequence also does not exist.  It suffices to consider the following five cases:\\
Case I: $y\in J^-(p)$\\
Case II: $y\in J^+(q)$\\
Case III: $y\notin I^+(p)\cup I^-(q)$\\
Case IV: $y\in I^-(q)\setminus J^-(p)$\\
Case V: $y\in I^+(p)\setminus J^+(q)$

To see this, if case III does not hold then $y\in I^+(p)$ or $y\in I^-(q)$ so one of the remaining four cases must hold.

\textbf{Case I:} Suppose $y\in J^-(p)$.  Then no null-geodesics through $y$ enter $U=I^+(p)\cap I^-(p)$.

Indeed, assume by way of contradiction that a null-geodesic through $y$ enters $U$.  Orient the null-geodesic so that it first passes through $y$ and then enters $U$.  Since $y\in J^-(p)$ the null-geodesic is future-directed.  If there exists a future-directed null-geodesic from $y$ to some point $a\in U$, then $d(y,a)=D(y,a)=0$.  Here $D$ is a local Lorentz distance defined on $V\times V$, see \cite[p. 160]{BEE}.  The first equality holds by \cite[Theorem 4.27]{BEE} and the second equality holds by the remarks below \cite[Definition 4.25]{BEE}.  We therefore have $y\leq p<a$ so $y<a$, thus by the same remarks in \cite{BEE} we have $d(y,a)>0$, a contradiction.

\textbf{Case II:} Suppose $y\in J^+(q)$.  Then no null-geodesics through $y$ can enter $U$.

Suppose by way of contradiction that a null-geodesic through $y$ enters $U$.  Orient the null-geodesic so that it first passes through $y$ and then through $U$.  Since $y\in J^+(q)$ the null-geodesic is past-directed.  If there exists a past-directed null-geodesic from $y$ to some point $a\in U$, then $d(a,y)=D(a,y)=0$ (as in case I) but $a<q\leq y$ so $a<y$, thus $d(a,y)>0$, a contradiction.

Note that this case is \emph{future-past analogous} to case I, in the sense that we obtain a symmetric proof by interchanging the roles of $p$ and $q$; $+$ and $-$; future and past.

\textbf{Case III:} Suppose $y\notin I^+(p)\cup I^-(q)$.  Then no null-geodesics through $y$ can enter U.

To see this, suppose there exists a future-directed null-geodesic from $y$ to $a\in U$.  Then $y\leq a<q$ so $y<q$, a contradiction.  On the other hand, suppose there exists a past-directed null-geodesic from $y$ to $a\in U$.  Then $p<a\leq y$ so $p<y$, a contradiction.

\textbf{Case IV:} If $y\in I^-(q)\setminus J^-(p)$ then not all null-geodesics through $y$ enter U.

Suppose by way of contradiction that all null-geodesics through $y$ enter $U$.  Consider a null-geodesic from $y$ to some point $a\in U$. This null-geodesic enters $U$ at a point $b\in J^+(p)\setminus I^+(p)$, see Proposition~\ref{3.4BEE}, \cite[Proposition 3.4]{BEE}.  There exists a null-geodesic from $p$ to $b$ by \cite[Proposition 3.4]{BEE}.  Extend it until it leaves $I^-(q)$. This null-geodesic must eventually leave by the compactness of $J^+(p)\cap J^-(q)$, which follows from the global hyperbolicity of $V$.  Again by \cite[Proposition 3.4]{BEE} this occurs at a point $c\in (J^+(p)\cap J^-(q))\setminus U$.  Then $y\leq c$ since $y\leq b$ and $b\leq c$.

This means $y<c$ since if there exists a null-geodesic from $y$ to $c$ then this null-geodesic cannot enter $U$, a contradiction.  The reason the null-geodesic from $y$ to $c$ cannot enter $U$ is that this cannot happen in the future of $c$ (otherwise we would have $c\in I^-(q)$ which is not the case).  Suppose the null-geodesic from $y$ to $c$ passes through $U$ after $y$ and before $c$.  If $y\in \partial U$ so that $y=b$ this is impossible since the null-geodesic from $p$ to $c$ never enters $U$.  Therefore it suffices to assume $y\notin \partial U$ in which case the null-geodesics from $y$ to $b$ and from $b$ to $c$ must be distinct.  However, by \cite[Lemma 2.16]{Penrose} this implies $c\in I^+(y)$ which contradicts the fact that there is a null-geodesic from $y$ to $c$.

On the other hand, we cannot have $y<c$ either.  It suffices to show there exists a null-geodesic from $y$ intersecting the null-geodesic from $q$ through $c$, say at some point $c'\leq q$.

We explain why this is sufficient to complete the proof.  Observe that $c'$ cannot lie in the past of $c$.  Indeed, assume $c'<c$.  Then the null-geodesic from $y$ to $c'$ never passes through $U$.  To see this, the null-geodesic cannot enter $U$ in the past of $c'$ since this would mean $c'\in I^+(p)$ and hence $c\in I^+(p)$, a contradiction.  On the other hand, the null-geodesic cannot enter $U$ in the future of $c'$ since this would imply $c'\in I^-(q)$, a contradiction.

Therefore $c'$ must lie strictly between $q$ and $c$ so we obtain a contradiction to the Lorentz triangle inequality: $y<c$ and $c\leq c'$ so $y<c'$, but there exists a null-geodesics from $y$ to $c'$.

We now turn to the existence of a null-geodesic from $y$ to such a point $c'$.  The open set $V$ is globally hyperbolic so $J^+(y)\cap J^-(q)$ is compact and contained in $V$.  Note that $q\in J^+(y)$ because $y\leq b\leq c\leq q$.  The past-directed null-geodesic from $q$ through $c$ must eventually leave the compact subset $J^+(y)\cap J^-(q)\subset V$, at a point $c'$.  By \cite[Proposition 3.4]{BEE} there exists a null-geodesic from $y$ to $c'$.

\textbf{Case V:} Suppose $y\in I^+(p)\setminus J^+(q)$.  This is future-past analogous to case IV.

\end{proof}

\begin{example}
Similarly if $(M,g)$ is a quotient of the unit sphere metric by a finite group of isometries (e.g. a three-dimensional Lens space), then the Lorentz product manifold $(M\times \mathbb{R},g\oplus-dt^2)$ is strongly refocusing at every event.
\end{example}

\begin{example}
We generalize the previous example.  A \emph{$Y^x_l$-manifold} is a complete connected Riemannian manifold all of whose unit speed geodesics starting at $x\in M$ return to $x$ at time $l$, see \cite[p. 181]{Besse}, \cite{CKS}.  If $(M,g)$ is a quotient of any \emph{$Y^x_l$-manifold} by a finite group of isometries, then the Lorentz product manifold $(M\times \mathbb{R},g\oplus-dt^2)$ is strongly refocusing at $(x,t)$ for all $t\in \mathbb{R}$, see~\cite[Remark 7]{CR}.
\end{example}

We prove a proposition which gives another characterization of refocusing that will be useful in later sections.

\begin{prop}\label{closedprop}
Suppose that a strongly causal space-time $(X,g)$ is refocusing at $x\in X$.  Then, given any causally convex, convex normal open neighborhood $V$ of $x$ with compact closure, for all smaller open neighborhoods $U$ of $x$ there exists a point $y\notin V$ such that all the null-geodesics  through $y$ enter $U$.  

In other words, the neighborhood $V$ in the definition of refocusing may be replaced with any neighborhood of $x$ which is causally convex and convex normal with compact closure.
\end{prop}

\begin{proof}
Since $(X,g)$ is refocusing at $x$, there exists an open neighborhood $W$ of $x$ such that for all smaller open neighborhoods $U$, there is an event $y\notin W$ such that all the null-geodesics through $y$ enter $U$.

Let $V$ be a causally convex and convex normal neighborhood of $x$ with compact closure, and let $U$ be any open neighborhood of $x$ contained in $V$.  We wish to show that there exists an event $y\notin V$ such that all the null-geodesics through $y$ enter $U$.

The set $U\cap W$ is an open neighborhood of $x$.  Furthermore, since $X$ is strongly causal the Alexandrov and manifold topologies of $x$ agree, see Theorem~\ref{KP} \cite[Theorem 4.24]{Penrose}.  Therefore we may choose an open neighborhood of $x$ of the form $N=I^+(p)\cap I^-(q)$ which is contained in $U\cap W$.

Since $N\subset W$, there is an event $y\notin W$ such that all of the null-geodesics through $y$ enter $N$.  Since $N=I^+(p)\cap I^-(q)\subset V$, the proof of (i.)$\Rightarrow$(iii.) of the equivalence of Low's three definitions guarantees that $y\notin V$.

Therefore, there exists an event $y\notin V$ such that all of the null-geodesics through $y$ enter $N$ and hence $U$, which completes the proof.
\end{proof}

\begin{definition}
Let $(X^{n+1},g)$ be a space-time, and let $\rho: \widetilde{X}\to X$ be a smooth covering map.  Then $\widetilde{X}$ is a smooth connected $(n+1)$-manifold.  The pullback $\widetilde{g}=\rho^*(g)$ defines a Lorentz metric on $\widetilde{X}$.  Let $\mathcal Y$ be a timelike vector field defining the time-orientation of $X$.  There exists a unique vector field $\widetilde{\mathcal{Y}}$ on $\widetilde{X}$ which projects to $\mathcal{Y}$, i.e.~$d\rho_p(\widetilde{\mathcal{Y}}_p)=Y_{\rho(p)}$ for all $p\in \widetilde{X}$, see \cite[Section 7.5]{BrickellClark}.  This vector field $\widetilde{\mathcal{Y}}$ is timelike and gives a time-orientation of $(\widetilde{X},\widetilde{g})$ making it into a space-time.  We call the covering map $\rho:(\widetilde{X},\widetilde{g})\to (X,g)$ a \emph{Lorentz covering} of $(X,g)$, and we call the space-time $(\widetilde{X},\widetilde{g})$ a \emph{Lorentz cover} of $(X,g)$.
 \end{definition}

We now prove a theorem which generalizes the result \cite[p. 345]{CR},~Theorem~\ref{ChernovFinite} of Chernov and Rudyak, proved for globally hyperbolic space-times.  Specifically, a Lorentz cover of a strongly causal refocusing space-time is also a strongly causal refocusing space-time.  We will need the following Lemma.

\begin{lemma}\label{coveringlemma}
Let $(X^{n+1},g)$ be a strongly causal space-time, $x,y\in X$ distinct points, and $U$ a causally convex neighborhood of $x$ such that all the null-geodesics passing through $y$ enter $U$.  Then there exists a smooth embedding $\Psi$ of $\mathbb{S}^{n-1}$ into the tangent future null-cone of $y$, consisting of exactly one vector $\Psi(v)$ in each future null direction $v$ at $y$, whose image under $exp_y$ is contained in $U$.
\end{lemma}

\begin{proof}
We assume that $y\notin U$ since the case $y\in U$ is obvious.  Since $U$ is causally convex, we may assume without loss of generality that all null-geodesics through $y$ enter $U$ in the future of $y$.  (If $\gamma_1$ and $\gamma_2$ are null-geodesics passing through $y$ which enter $U$ in the future and in the past of $y$ respectively, then together $\gamma_1\cup \gamma_2$ form a nonspacelike curve intersecting $U$ along more than one connected interval.)

Let $V$ be the open subset of the tangent null-cone of $y$ on which $exp_y$ is defined.  Put $W=(exp_y|_V)^{-1}(U)$.  Then $W$ is an open subset of $V$ (and hence an open subset of the future tangent null-cone) by the continuity of $exp_y|_V$.  The future tangent null-cone of $y$ is diffeomorphic with $\mathbb{S}^{n-1}\times (0,\infty)$ via a fiber preserving diffeomorphism, so we may identify $W$ with an open subset of $\mathbb{S}^{n-1}\times (0,\infty)$.  By the causal convexity of $U$, if $(v,a),(v,b)\in W$ and $a<b$ then $(v,c)\in W$ for all $c\in [a,b]$.

We define a function $f: \mathbb{S}^{n-1}\to (0,\infty)$ (not necessarily continuous) by assigning to each future null direction $v$ at $y$ a positive real number $t\in (0,\infty)$ such that $(v,t)\in W$.  Note that the collection of open subsets of $W$ of the form $B\times I$, where $B$ is open and $I$ is an open interval, form a base for $W$.  We will call these base elements \emph{cylinders}.  Associate to each point $(v,f(v))$ a cylinder in $W$ containing it.  The coordinate projection onto $\mathbb{S}^{n-1}$ carries each cylinder to an open set in $\mathbb{S}^{n-1}$.  Furthermore, these open sets cover $\mathbb{S}^{n-1}$.  Therefore by the compactness of $\mathbb{S}^{n-1}$ there exists a finite subcover of $\mathbb{S}^{n-1}$ by open sets $U_1,...,U_k$.  There exists a smooth partition of unity on $\mathbb{S}^{n-1}$ which is subordinate to this cover.  Therefore we have nonnegative smooth functions $f_1,...,f_k:\mathbb{S}^{n-1}\to \mathbb{R}$ such that $supp(f_i)\subset U_i, i=1,...,k$ and $\Sigma_{i=1}^k f_i(v)=1$ for all $v\in \mathbb{S}^{n-1}$.

Define a function $\phi: \mathbb{S}^{n-1}\to (0,\infty): v\mapsto \Sigma_{i=1}^k f_i(v)f(v_i)$, where $v_i\in U_i\subset \mathbb{S}^{n-1}$ and $f: \mathbb{S}^{n-1}\to \mathbb{R}$ is defined as above.  Now define $\Psi: \mathbb{S}^{n-1}\to W: v\mapsto (v,\phi(v))$.  We complete the proof of the lemma by showing that $\Psi$ is a smooth embedding.

First, observe that $\Psi(\mathbb{S}^{n-1})\subset W$.  To see this, let $v\in \mathbb{S}^{n-1}$.  The union of the cylinders over $U_1,...,U_k$ which contain $v$ is contained in $W$.  Thus $(v,f(v_i))\in W$ for each $i$ such that $v\in U_i$.  Let $m$ and $M$ be the minimum and maximum values, respectively, of $f(x_i)$ taken over all $i$ such that $v\in U_i$.  Therefore $(v,c)\in W$ for $m\leq c\leq M$.
Now, $m=\Sigma_{i=1}^k f_i(v)m\leq \Sigma_{i=1}^k f_i(v)f(v_i)\leq \Sigma_{i=1}^k f_i(v)M=M$ so that $m\leq \phi(v)\leq M$ and hence so $\Psi(v)\in W$.

The map $\Psi: \mathbb{S}^{n-1}\to W: v\mapsto (v,\phi(v))$ is smooth in the first coordinate because the first coordinate projection is the identity map.  It is also smooth in the second coordinate, i.e.~the map $\phi$ is smooth.  To see this, let $v\in \mathbb{S}^{n-1}$ and let $U_v$ be the intersection of  those sets $U_i$ such that $v\in U_i$.  Thus $U_v$ is a nonempty open neighborhood of $v$ in $\mathbb{S}^{n-1}$.  The restriction $\phi|_{U_v}$ is a smooth map because it is a linear combination of the smooth functions $f_1,...,f_k$ with coefficients defined by $c_i=f(v_i)$ if $v\in U_i$ and $c_i=0$ if $v\notin U_i$. Thus $\phi$, and hence $\Psi$, is a smooth map.  herefore $\Psi$ is a smooth bijection of $\mathbb{S}^{n-1}$ onto its image in $W$.  The inverse map $\Psi^{-1}$ is also smooth because it is the first coordinate projection.  This completes the proof of the lemma.
\end{proof}

\begin{theorem}\label{coveringtheorem}
Let $(X^{n+1},g)$, $n\geq 2$ be a strongly causal space-time which is refocusing at $x\in X$, and let $\rho: (\widetilde{X},\widetilde{g})\to (X,g)$ a Lorentz covering map.  Then $(\widetilde{X}^{n+1},\widetilde{g})$ is a strongly causal space-time which is refocusing at $\widetilde{x}\in \widetilde{X}$ for all $\widetilde{x}\in \rho^{-1}(x)$. Moreover if $(X^{n+1},g)$ is strongly causal and strongly refocusing at $x$, then 
 $(\widetilde{X}^{n+1},\widetilde{g})$ is strongly causal and strongly refocusing at $\widetilde{x}\in \widetilde{X}$ for all $\widetilde{x}\in \rho^{-1}(x)$.
\end{theorem}

\begin{proof}
We show first that $(\widetilde{X},\widetilde{g})$ is strongly causal.  By way of contradiction, assume that $(\widetilde{X},\widetilde{g})$ fails to be strongly causal at some point $\widetilde{q}\in \widetilde{X}$.  Thus, there exists an open neighborhood $V$ of $\widetilde{q}$ such that for all open neighborhoods $U$ of $\widetilde{q}$ with $U\subset V$, there exists a nonspacelike curve $\gamma$ which leaves $U$ and then returns to $U$.  Note that we may replace $V$ with any open subset of $V$ containing $\widetilde{q}$.  Thus we may assume that $V$ is contained in a single component of the preimage of a neighborhood of $q=\rho(\widetilde{q})\in X$ which is evenly covered by $\rho$.  Take $U\subset V$ such that $\rho(U)$ is a causally convex neighborhood of $q$.  Now, $\rho$ is a covering map and hence open, so $\rho(U)$ and $\rho(V)$ are open neighborhoods of $q\in X$ such that $\rho(U)\subset \rho(V)$.  Furthermore, $\rho \circ \gamma$ is a nonspacelike curve in $X$ which leaves $\rho(U)$ and then returns.  Thus $(X,g)$ fails to be strongly causal at $q\in X$, a contradiction.

Take any $\widetilde{x}\in \rho^{-1}(x)$.  Take $V$ to be an open neighborhood of $x\in X$ as in the definition of $(X,g)$ being refocusing at $x$.  We may assume that $V$ is evenly covered by $\rho$.  Let $U$ be any causally convex open neighborhood of $x$ which is contained in $V$.  Take $y\in X\setminus V$ such that all of the null-geodesics in $(X,g)$ which pass through $y$ enter $U$.

Given any null-geodesic $\gamma$ in $X$ starting at $y$, i.e. $\gamma(0)=y$, choose the null vector $\Psi(\mathbf{v})\in W\subset T_yX$ in the direction of $\gamma'(0)$ defined by the map $\Psi$ of the preceding Lemma~\ref{coveringlemma}.  Thus $z=exp_y(\Psi(\mathbf{v}))\in U$.  Let $\widetilde{V}$ be the connected component of $\rho^{-1}(V)$ containing $\widetilde{x}\in \rho^{-1}(x)$.  Put $\widetilde{z}$ to be the unique point of $\rho^{-1}(z)$ located within $\widetilde{V}$.  Let $\widetilde{\gamma}$ be the unique lift of $\gamma$ such that $\widetilde{\gamma}(\bar{t})=\widetilde{z}$, where $\bar{t}$ is given by $\gamma(\bar{t})=z$.  Put $\widetilde{y}=\widetilde{\gamma}(0)$ so that $\rho(\widetilde{y})=y$.

Now let $\widetilde{U}\subset \widetilde{V}$ be any open neighborhood of $\widetilde{x}$ which is contained in $\widetilde{V}$.  Without loss of generality, we may choose $U\subset X$ as above so that $\rho^{-1}(U)\cap \widetilde{V}\subset \widetilde{U}$.

We show that all null-geodesics in $\widetilde{X}$ passing through $\widetilde{y}$ enter $\widetilde{U}$.  We may identify the tangent spaces $T_yX$ and $T_{\widetilde{y}}\widetilde{X}$ because $\rho$ is a local diffeomorphism and hence $d\rho_{\widetilde{y}}: T_{\widetilde{y}}\widetilde{X}\to T_{y}X$ is a linear isomorphism.  Since $\Psi(\mathbb{S}^{n-1})$ is connected and $exp_{\widetilde{y}}$ is continuous, $exp_{\widetilde{y}}(\Psi(\mathbb{S}^{n-1}))$ is a connected subset of $\rho^{-1}(V)$ and therefore lies in a single connected component of $\rho^{-1}(V)$.  Furthermore, $exp_{\widetilde{y}}(\Psi(\mathbb{S}^{n-1}))$ contains $\widetilde{z}$ due to the null-geodesic $\widetilde \gamma$ we constructed above.  Thus $exp_{\widetilde{y}}(\Psi(\mathbb{S}^{n-1}))$ is a subset of $\widetilde{V}$ and hence a subset of $\widetilde{U}$, since $exp_{\widetilde{y}}(\Psi(\mathbb{S}^{n-1}))\subset \rho^{-1}(U)$.  This completes the proof of the theorem for the refocusing case. The proof for the strongly refocusing case is similar and in fact easier.
\end{proof}

We conclude this section by mentioning two important theorems.  These theorems show that refocusing imposes strong topological restrictions on the Cauchy surface of a globally hyperbolic space-time.

\begin{theorem}\label{LowCompact}
Low, \cite[Theorem 5]{LowRefocusing} A Cauchy surface of a globally hyperbolic refocusing space-time is compact.
\end{theorem}

\begin{theorem}\label{ChernovFinite}
Chernov-Rudyak, \cite[Theorem 14]{CR}.  A Lorentz cover of a globally hyperbolic refocusing space-time is also globally hyperbolic and refocusing, with respect to the Lorentz metric induced by the covering map.

Furthermore, a Cauchy surface of a globally hyperbolic refocusing space-time has finite fundamental group.
\end{theorem}

\section{A Homeomorphism of a space-time with its Sky Space}

In determining suitable conditions under which a space-time and its sky space are diffeomorphic via the map sending each point to its own sky, Robert J. Low \cite[p. 5]{LowCones}, \cite[p. 6]{LowNullgeodesics}, \cite[p. 46]{LowRefocusing} introduced the three definitions of \emph{weak refocusing}, also called \emph{refocusing}.  These definitions were given in Definition~\ref{refocusing} and we proved that they are equivalent.  

Low observed that for a strongly causal, non-refocusing space-time, the map to its sky space which sends each event to its sky is a diffeomorphism.  Low \cite[p. 6]{LowNullgeodesics} states that this map is a homeomorphism and proceeds to proving that it is a diffeomorphism.  He did not address the question why this is a homeomorphism in his work.  Low encouraged me to include the formal proof of this statement into my dissertation.

First we will need the following proposition which states that this map is indeed a homeomorphism for globally hyperbolic non-refocusing space-times.  The thorough sketch of the proof of Proposition~\ref{mainprop2} was provided to me by Vladimir Chernov.

\begin{prop}\label{mainprop2}
If $(X,g)$ is a globally hyperbolic non-refocusing space-time, then $X$ is homeomorphic with its sky space $SKY(X)$ via the map $f:X\rightarrow SKY(X): x\mapsto O_x$.
\end{prop}

\begin{proof}
Clearly $f$ is surjective.  If $x\neq y$ but $O_x=O_y$, then all null-geodesics through $y$ pass through $x$, so $X$ is strongly refocusing and hence refocusing, a contradiction.  Thus $f$ is bijective.

To show $f$ is an open map, it suffices to show that for each $x\in X$, there are arbitrarily small open neighborhoods $U$ of $x$ such that $f(U)$ is open in $SKY(X)$.
Choose arbitrarily small $U$ as in the definition of non-refocusing, i.e. for all $y\notin U$, there exists a null-geodesic through $y$ not entering $U$.  Since $X$ is globally hyperbolic, $X$ is diffeomorphic to $M\times \mathbb{R}$ where each $M\times \{t\}$ is a smooth spacelike Cauchy surface.  For each $t\in \mathbb{R}$, let $M_t=M\times \{t\}$ and let $\pi_t:STM_t\rightarrow M_t$ be the corresponding fibration.  Since $\pi_t$ is continuous, the set $W=\bigcup_{t\in \mathbb{R}} \pi_t^{-1}(U\cap M_t)$ is open in $\mathcal{N}\cong STM$, identifying the $STM_t$'s with $STM$ for a fixed $M=M\times \{t\}$.  The set $O$ of all skies contained in $W$ is an open subset of $SKY(X)$, namely the base element at $O_x$ in $SKY(X)$ defined by $W$.  Thus, it suffices to show $f(U)=O$.  We show both containments.\\
$\mathbf{(i.)}O\subset f(U)$, i.e. no points outside of $U$ are mapped to $O$.  Suppose by way of contradiction that $y\notin U$ but $O_y\in O$, so that each null-geodesic through $y$ is contained in $W$, thus it's contained in $\pi_t^{-1}(U\cap M_t)$ for some $t$, and so passes through $U$, a contradiction.\\
$\mathbf{(ii.)}f(U)\subset O$.  Suppose $y\in U$.  By definition, $O$ is the set of all skies contained in $W$, so we need only show $O_y\subset W$.  Since $y\in U$, we have $y\in U\cap M_t$ for some $t$, so 
$O_y=\pi_t^{-1}(y)\in \pi_t^{-1}(U\cap M_t)$, hence $O_y\subset W$.

It remains to show that $f$ is continuous.  Let $x\in X$, and let $U$ be an open neighborhood of $O_x$ in $SKY(X)$. We show that $f^{-1}(U)$ is open in $X$, i.e. for all $y\in f^{-1}(U)$, there exists an open neighborhood $V$ of y such that $f(V)\subset U$.  It suffices to assume that $U$ is a base element at $O_x$.  Thus, $U$ is the set of all skies contained in $W$, where $W$ is an open subset of $\mathcal{N}\cong STM$ containing $O_x$.  Here $M$ is a smooth spacelike Cauchy surface of $X\cong M\times \mathbb{R}$.  Without loss of generality, $M=M_0$ is the Cauchy surface containing $y$.  We denote the projection by $\pi=\pi_0: STM\to M$.  We have $f^{-1}(U)=\{z:O_z\subset W\}$, and $O_y=\pi^{-1}(y)$ is an $(n-1)$-sphere in $STM$ where $dim(X)=n+1$.  We show that there exists an open neighborhood $Z$ of $y$ in $M$ with $\pi^{-1}(Z)\subset W$.

Observe that there exists an open neighborhood of $E$ of $O_y$ in $\mathcal{N}\cong STM$ such that $E\cong (B^n)\times \mathbb{S}^{n-1}$ where $B^n$ denotes an open $n$-ball.  This holds by the local trivialization of the spherical tangent bundle $STM$.  If $E\subset W$ then we put $Z=B^n$ so that $\pi^{-1}(Z)\subset W$.  Thus without loss of generality we may assume $E\setminus W\neq \phi$.  Define a distance function $d$ on $E$ by $d=\sqrt{(d_{g|_{M}})^2+m^2}$ where $d_{g|_{M}}$ is the metric on the open ball $B^n$ coming from $(M,g|_{M})$ and $m$ is the standard Riemannian metric on the Sphere.  Define $G: E\to \mathbb{R}$ by $G(p)=inf\{d(p,q): q\in E\setminus W\}$.  Now $G(p)>0$ for all $p\in \pi^{-1}(y)$ since $E\setminus W$ is a closed subset of $E$, so if $G(p)=0$ then $p\in E\setminus W$, a contradiction.  The function $G$ is continuous and the set $\pi^{-1}(y)$ is a compact subset of $E$ being a sphere, hence by the extreme value theorem $G$ obtains a minimum value $A>0$ on $\pi^{-1}(y)$.  Then $Z=B_A(y)\subset M$, the ball of radius $A$ with respect to $d_{g|_{M}}$ centered at $y$, is an open neighborhood of $y$ in $M$ satisfying $\pi^{-1}(Z)\subset W$ as desired.

Identify $y\in M=M_{0}\subset X$ with $(\hat{y},0)\in X\cong M\times \mathbb{R}$.  We may choose $t>0$ sufficiently small so that $(\hat{y},t)\in X\cong M\times \mathbb{R}$ satisfies:\\
\textbf{(i.)}$exp_{(\hat{y},t)}|_{exp_{(\hat{y},t)}^{-1}(C^+)}$ is an embedding, where the set $C^+\subset X\cong M\times \mathbb{R}$ is defined by $C^+= exp_{(\hat{y},t)}((PD_{(\hat{y},t)})\cap (M\times (0,\infty)))$, and $PD_{(\hat{y},t)}$ denotes the open subset of the past-directed nonspacelike cone of $(\hat{y},t)$ on which the exponential map is defined.  To see that this can be done, consider a convex neighborhood of $(\hat{y},t)$.\\
\textbf{(ii.)}Each point $p\in C^+$ has its own such cone $exp_p((PD_p)\cap (M\times (0,\infty)))$ contained in $C^+$.\\
\textbf{(iii.)}$exp_{(\hat{y},t)}((PD_{(\hat{y},t)})\cap (M\times 0))\subset Z$.

We define the set $C^-\subset X$ to be future-past analogous to $C^+$.  Now $M\cap C^+\cap C^-$ is an open neighborhood of $\hat{y}$ in $M$, so $V=D(M\cap C^+\cap C^-)$ is an open neighborhood of $y$ in $X$ which is contained in the union $C^+\cup C^-$.  (Here $D$ denotes the Cauchy development, see Definition~\ref{CauchyDevelopment}.) Consider the sky of any point in the interior of $V$.  Since $\pi^{-1}(Z)\subset W$ the intersection of the null-cone of the point with $M$ is contained in $Z$, hence the sky is contained in $W$.  It follows that the open neighborhood $V=D(M\cap C^+\cap C^-)$ of $y$ in $X$ satisfies the desired property $f(V)\subset U$.  This completes the proof of continuity.
\end{proof}

We also will need the following lemma.

\begin{lemma}\label{homeomorphismlemma}
Let $(X,g)$ be a strongly causal space-time.  Suppose that $x\in X$ and $U\ni x$ is a causally convex open neighborhood of $x$, that exists by Proposition~\ref{mainprop1}.  Let $\mathcal{N}_U$ be the space of all future-directed null-geodesics in $U$ (modulo orientation-preserving affine reparametrizations).

Then there is a one-to-one correspondence between $\mathcal{N}_U$ and the subset of $\mathcal{N}_X$ formed by those null-geodesics in $X$ which pass through $U$.  Furthermore, under this identification $\mathcal{N}_U$ is an open subset of $\mathcal{N}_X$.
\end{lemma}

\begin{proof}
The causal convexity of $U$ gives a one-to-one correspondence between the set $\mathcal{N}_U$ of null-geodesics in $U$ and the set of null-geodesics in $X$ which enter $U$.  Thus we may identify $\mathcal{N}_U$  with a subset of $\mathcal{N}_X$.

Recall that $SNX$ is the quotient of the future null bundle $FNX$ by the group action of $\mathbb{R}^+$, and $\mathcal{N}_X$ is the quotient of $SNX$ by the equivalence relation identifying two null-directions which belong to a common null-geodesic.  We want to show that the preimage of $\mathcal{N}_U\subset \mathcal{N}_X$ under the natural projection $SNX\to \mathcal{N}_X$ is an open subset of $SNX$.  By the definition of the quotient topology it suffices to show that the preimage $V$ of $\mathcal{N}_U$ under the composition of quotient maps $FNX\to SNX\to \mathcal{N}_X$ is an open subset of $FNX$.

We show that $FNX\setminus V$ is closed in $FNX$.  $FNX\setminus V$ is the set of all future-directed null vectors in $X$ which are the velocity vectors of geodesics not passing through $U$.  Suppose, by way of contradiction, that a sequence $\{\mathbf{v_n}\}_{n=1}^\infty$ of vectors in $FNX\setminus V$ converges to a vector $\mathbf{v}\in V$.  Note that $exp(l\mathbf{v})\in U$ for some $l\in \mathbb{R}$.  Then $exp(l\mathbf{v_n})\rightarrow exp(l\mathbf{v})\in U$, a contradiction because a sequence of points not in $U$ cannot converge to a point in $U$ (since $X\setminus U$ is closed in $X$).  Note that $exp(l\mathbf{v_n})$ is defined for sufficiently large $n$ because the exponential map is defined on an open subset of $TX$.
\end{proof}

\begin{theorem}\label{maintheorem2}
If $X$ is a strongly causal non-refocusing space-time, then the map $f:X\to SKY(X): x\mapsto O_x$ is a homeomorphism.
\end{theorem}

\begin{proof}
Clearly $f$ is surjective.  $f$ is injective because $X$ is non-refocusing.  (If $O_x=O_y$ for $x\neq y$ then all null-geodesics through $x$ pass through $y$, contradicting non-refocusing.) We show \textbf{(I.) $f$ is an open map} and \textbf{(II.) $f$ is continuous.}  Clearly this proves the theorem.
\\

\textbf{(I.)$f$ is an open map.}  Let $Y$ be an open subset of $X$.  We show $f(Y)=\{O_y: y\in Y\}$ is an open neighborhood of each of its points in $SKY(X)$.  Take any $x\in Y$.  Because $X$ is strongly causal, by Proposition~\ref{mainprop1} there exists a globally hyperbolic open neighborhood $V$ of $x$ contained in $Y$ which is convex normal and causally convex and has compact closure.  Since $X$ is non-refocusing, there is an open neighborhood $U'$ of $x$ contained in $V$ (and hence contained in $Y$) such that for every $y\notin U'$, not all null-geodesics through $y$ enter $U'$.  Let $U=I^+(p)\cap I^-(q)\subset U'$ for $p,q\in V$.  (Such a neighborhood $U$ exists by Theorem~\ref{KP} \cite[Theorem 4.24]{Penrose}.)  Then for all $y\notin U$, not all null-geodesics through $y$ enter $U$.  To see this, if $y\notin V$ then not all null-geodesics enter $U\subset U'$ because not all geodesics through $y$ enter $U'$.  If $y\in V\setminus U$ then not all null-geodesics through $y$ enter $U$ by the proof of Proposition~\ref{maintheorem1}.

We show that $f(U)$ (and hence $f(Y)$) contains a neighborhood base element $B$ at $O_x$.  Let $\mathcal{N}_X$ be the space of all future-directed null-geodesics in $X$ modulo orientation-preserving affine reparametrizations.  Let $W=\mathcal{N}_U$ be the subset of $\mathcal{N}_X$ formed by those null-geodesics in $X$ which enter $U$.  By Lemma~\ref{homeomorphismlemma} $W=\mathcal{N}_U$ is an open subset of $\mathcal{N}_X$ since $U$ is causally convex.  (If $x,z\in U$ and $x\leq y\leq z$ then $p<x \leq y \leq z<q$ so $y\in U$.)  Let $B$ be the base element of $SKY(X)$ at $O_x$ consisting of all skies contained in $W$.  Then $B\subset f(U)$.  To see this, suppose $O_y\in B$ but $y \notin U$.  Not all null-geodesics through $y$ can enter $U$.  Thus, $O_y$ cannot belong to $f(U)$, a contradiction.  Therefore $f(Y)$ is a neighborhood of each of its points.  This shows that $f$ is an open map.
\\
\textbf{(II.)$f$ is continuous.}  Let $x\in X$.  We show that $f$ is continuous at $x$.  It suffices to show that the restriction of $f$ to an open neighborhood of $x$ is continuous.

Define $V$ and $U$ as in part I above.  A sky in $SKY(U)$ cannot be the sky of an element $y\notin U$ because $U$ is chosen to satisfy the definition of non-refocusing.  Thus $SKY(U)=\{O_y: y\in U\}$ may be identified with the subset of $SKY(X)$ formed by the skies of points in $U$.

By Proposition~\ref{mainprop2} the map $F: U\to SKY(U)$ given by restriction of the domain and codomain of $f: X\to SKY(X)$ is continuous because $U$ is globally hyperbolic.  The map $f$ is therefore continuous because it is the composite of continuous maps $f=i \circ F$ where $i: SKY(U)\to SKY(X)$ is the inclusion.  The map $i$ is indeed the inclusion map because the sky space $SKY(U)$ of $U$ has the subspace topology coming from $SKY(X)$.  To see this, let $\mathcal{N}_U$ be the space of null-geodesics in $U$ modulo orientation-preserving affine reparametrizations.  For $x\in U$ a $U$-base at $O_x$ (for $SKY(U)$ as a sky space) is the set of all skies in $SKY(U)$ contained in an open subset $W$ of $\mathcal{N}_U$ containing $O_x$.  This is exactly the same as the subspace topology: an $X$-base at $O_x$ (for $SKY(U)$ as a subspace) is the set of all skies in $SKY(U)$ contained in an open subset $W'$ of $\mathcal{N}_X$ containing $O_x$.  These two topologies on $SKY(U)$ coincide because $\mathcal{N}_U$ is an open subset of $\mathcal{N}_X$ by Lemma~\ref{homeomorphismlemma}.
\end{proof}

\section{Weak but Not Strong Refocusing at a Point or on a Hypersurface}

Strong refocusing implies refocusing.  The converse is not clear, even for globally hyperbolic space-times.  This is one of the central problems which motivated the results of this work.

We begin the investigation of whether or not refocusing implies strong refocusing.  We start by constructing the first known examples of space-times which are refocusing at an event $p$ but not strongly refocusing at $p$. Some of these space-times are globally hyperbolic and some of them have this property at all points $p$ of a hypersurface (codimension-one submanifold).

\begin{example}
\emph{Refocusing but not strong refocusing at a point.}  Removing two points from the Einstein Cylinder, $(\mathbb{S}^n\times \mathbb{R}\setminus \{(-p,-\pi),(-p,\pi)\},m\oplus -dt^2)$ is a strongly causal space-time which is refocusing at $(p,0)$ but not strongly refocusing at $(p,0)$.  However, this space-time is still strongly refocusing because it is strongly refocusing at all other points.  
(Here $m$ denotes the standard Riemannian metric on the unit sphere $\mathbb{S}^n\subset \mathbb{R}^{n+1}$.)
\end{example}

\begin{example}
\emph{Refocusing but not strong refocusing at all points along a curve.}  
Take $c>0$ and $p\in \mathbb{S}^1.$  
Remove two vertical slits $p\times ((-\infty,0]\cup [c,\infty))$ and $-p\times((-\infty,\pi]\cup [\pi+c,\infty))$ from the Eisntein cylinder  $(\mathbb{S}^1\times \mathbb{R}, m\oplus-dt^2)$. The result is a strongly causal space-time which is refocusing but not strongly refocusing at all points along a curve defining the boundary of the region where strong refocusing occurs.  This region is disconnected if $c\leq 2\pi$ and connected if $c>2\pi$.
\end{example}

\begin{example}\label{ghexample}
\emph{Refocusing but not strong refocusing at all points along a hypersurface.}  Take $c$ such that $0<c<2\pi.$  The space-time $(\mathbb{S}^n\times(0,c),m\oplus-dt^2)$ is the globally hyperbolic space-time obtained by restricting the Einstein Cylinder metric on $\mathbb{S}^n\times \mathbb{R}$ to the open subset $\mathbb{S}^n\times (0,c)$.  This space-time is:\\
(a.) strongly refocusing at all points in the Cauchy surfaces $\mathbb{S}^n\times \{t\}$ satisfying $t<c-\pi$ or $\pi<t$.\\
(b.) refocusing but not strongly refocusing at all points in the Cauchy surfaces $\mathbb{S}^n\times \{t\}$ satisfying $t=\pi$ or $t=c-\pi$.\\
(c.) it is neither refocusing nor strongly refocusing elsewhere.
\end{example}

\begin{prop}\label{conformalprop}
Let $(X,g)$ be a strongly causal space-time which is refocusing at $x\in X$ and suppose $h$ is a Lorentz metric which is conformal to $g$, i.e. $h=\Omega g$ for a smooth positive function $\Omega: X\to (0,\infty)$.  Then $(X,h)$ is also refocusing at $x$.  In other words, refocusing is preserved by a conformal change of the metric.
\end{prop}

\begin{proof}
By \cite[Lemma 9.17]{BEE} the image of a null-geodesic under a conformal change is a null-geodesic (with proper reparametrization).  Thus for each null-geodesic $\gamma$ its image $Im(\gamma)\subset X$ is preserved by conformal change.  Therefore $(X,g)$ and $(X,h)$ have the same refocusing properties.
\end{proof}

\begin{remark}
All the examples we have discussed of space-times which are refocusing but not strongly refocusing at a particular point are geodesically incomplete.  In fact, each example above is timelike incomplete, null incomplete, and spacelike incomplete.  On the first glance it seems that these examples are refocusing but not strongly refocusing at a point (or set of points) because they are null incomplete.  The following proposition guarantees that there are also examples which are null complete.
\end{remark}

\begin{prop}\label{nulltimelikecompleteprop}
There exist globally hyperbolic space-times which are null and timelike complete, and refocusing but not strongly refocusing at a point, or at all points of a hypersurface.
\end{prop}

\begin{proof}
By \cite[Theorem 6.5]{BEE}, a strongly causal space-time $(X,g)$ can be made null and timelike complete by a conformal change.  Thus the examples above may be changed by a conformal factor to be null and timelike complete while having all the same refocusing properties by Proposition~\ref{conformalprop}.
\end{proof}

These examples led us to consider the properties of the set of points at which a strongly causal space-time is refocusing.  The preceding example shows that the set of points at which a strongly causal space-time is strongly refocusing need not be closed.  However, the set of points at which a strongly causal space-time is refocusing is closed.

\begin{theorem}\label{closedtheorem}
Let $(X,g)$ be a strongly causal space-time.  The possibly empty set of events at which $(X,g)$ is refocusing is closed.
\end{theorem}

\begin{proof}
 Suppose that $(X,g)$ is refocusing at each event of the sequence $\{x_n\}_{n=1}^\infty$ and suppose that $x_n\rightarrow x\in X$.  We show that $(X,g)$ is refocusing at $x$.

Let $V$ be a causally convex and convex normal open neighborhood of $x$ with compact closure.  Let $U$ be an open neighborhood of $x$ contained in $V$.  Choose $K\in \mathbb{N}$ sufficiently large so that $x_K\in U$.  Since the Alexandrov and manifold topologies agree for a strongly causal space-time by Theorem~\ref{KP} \cite[Theorem 4.24]{Penrose}, there exists an open neighborhood $N$ of $x_k$ of the form $N=I^+(p)\cap I^-(q)$ such that $N\subset U$.  Since $(X,g)$ is refocusing at $x_K$, and since $V$ is also a convex normal neighborhood of $x_K$,  there exists an event $y\notin V$ such that all the null-geodesics through $y$ enter $N$ and hence $U$, by Proposition~\ref{closedprop}.  This completes the proof.
\end{proof}

\section{Weak Versus Strong Refocusing}
We continue the discussion on relations between refocusing and strong refocusing.  We do not know examples of space-times that are refocusing but not strongly refocusing.  In this section we prove that in dimensions $\leq 4$ each globally hyperbolic refocusing space-time admits a metric with respect to which it is strongly refocusing. This suggests the possibility that the result may hold in all dimensions.

This nice fact discussed below follows immediately from the Geometrization conjecture proved by Perelman~\cite{Perelman1, Perelman2, Perelman3}, and from the results
of Low~\cite{LowNullgeodesics} and Chernov and Rudyak~\cite[Theorem 14]{CR}. 
The proof is a combination of the discussions in~\cite[Section 11, Remark 7]{CR} 
and in the works of Chernov and
Nemirovski~\cite[page 1321-1322]{ChernovNemirovskiGAFA}, \cite[Example 10.5]{ChernovNemirovskiGT}.  A related discussion is contained in our paper with Chernov and Sadykov, \cite[Section 7.2]{CKS}.

\begin{theorem}\label{maintheorem4}
Every globally hyperbolic refocusing space-time of dimension $3$ or $4$ admits a globally hyperbolic strongly refocusing metric (with strong refocusing at every event).  
\end{theorem}

\begin{proof}
If $(X^3,g)$ is a globally hyperbolic refocusing space-time, we have a diffeomorphism $X\cong M\times \mathbb{R}$ where $M^2$ is a Cauchy surface, which is compact by Theorem~\ref{LowCompact} \cite[Theorem 5]{LowRefocusing}, and has finite fundamental group by Theorem \ref{ChernovFinite} \cite[p. 345]{CR}.  Therefore by the classification theorem for closed surfaces, $M\cong \mathbb{S}^2$ if $M$ is orientable and $M\cong \mathbb{R}P^2$ if $M$ is not orientable.  Therefore $X\cong \mathbb{S}^2\times \mathbb{R}$ or $X\cong \mathbb{R}P^2\times \mathbb{R}$.

In either case, this diffeomorphism induces a Lorentz metric $g'$ on $X$ which is strongly refocusing at every event.  Specifically, if $M\cong \mathbb{S}^2$ then $g'$ is the metric coming from the Einstein cylinder metric $m\oplus -dt^2$ on $\mathbb{S}^2\times \mathbb{R}$ (where $m$ is the standard Riemannian metric on the unit $2$-sphere).  If $M\cong \mathbb{R}P^2$ then $g'$ is the metric coming from the Lorentz metric $m' \oplus -dt^2$ on $\mathbb{R}P^2\times \mathbb{R}$ (where $m'$ is the quotient of the unit sphere metric).

Assume $(X^4,g)$ is globally hyperbolic and refocusing.  Let $p: X_1\rightarrow X$ be a universal covering map, and equip $X_1$ with the induced Lorentz metric $g_1$ coming from $g$ via the local diffeomorphism $p$.  Then $M_1=p^{-1}(M)$ is a Cauchy surface in $X_1$, where $M$ is a Cauchy surface of $X$, so $(X_1,g_1)$ is globally hyperbolic.  Also $(X_1,g_1)$ is refocusing, see Theorem~\ref{ChernovFinite} \cite[p. 345]{CR}.

Therefore $M_1^3$ is compact by Theorem~\ref{LowCompact},  \cite[Theorem 5]{LowRefocusing} and also simply connected, so $M_1$ is diffeomorphic to $\mathbb{S}^3$ by the Poincar\'e conjecture, proved by Perelman in \cite{Perelman1}, \cite{Perelman2} and \cite{Perelman3}.  Therefore $X_1$ is diffeomorphic to $\mathbb{S}^3\times \mathbb{R}$, so $X_1$ admits a metric which is strongly refocusing at every event, namely the metric induced from $m\oplus -dt^2$ by the diffeomorphism, where $m$ is the standard Riemannian metric on the unit sphere $\mathbb{S}^3\subset \mathbb{R}^4$.

Now because $M$ is compact and has finite fundamental group by Theorem~\ref{ChernovFinite}  \cite[p. 345]{CR}, the Thurston Elliptization conjecture \cite{Perelman1}, \cite{Perelman2}, \cite{Perelman3} guarantees that $M$ is \emph{spherical}, i.e. a quotient of $\mathbb{S}^3\subset \mathbb{R}^4$ by a finite subgroup of the special orthogonal group $SO(4)$.  Therefore the group of deck transformations of the covering map $p: M_1\to M$ is identified with a finite subgroup of $SO(4)$.  Thus $M$ is equipped with a quotient metric $\bar{m}$ of the round sphere metric.  Now $(M\times \mathbb{R},\bar{m}\oplus -dt^2)$ is strongly refocusing at all events.
\end{proof}

\begin{remark}
Suppose that $(X^{n+1},g)$ is a globally hyperbolic space-time which is refocusing at $x\in X$.  By definition (ii.) of refocusing, there is an open neighborhood $V$ of $x$ such that for all \emph{smaller} open neighborhoods $U$ of $x$ ($U\subset V$), there exists an event $y\notin V$ such that all null-geodesics through $y$ enter $U$.  Let $U_n$ be a countable decreasing neighborhood base at $x$.

We then have a sequence of points in $\{y_n\}_{n=1}^\infty$ in $X\setminus V$ such that all null-geodesics through $y_n$ enter $U_n$.  Since $X$ is assumed to be globally hyperbolic, $X$ is diffeomorphic with $M\times \mathbb{R}$ where each $M\times \{t\}$ is a smooth spacelike Cauchy surface.  For convenience, write $y_n=(\hat{y}_n,t_n)$ and $x=(\hat{x},t)$ in $M\times \mathbb{R}$.  By Theorem~\ref{LowCompact},  \cite[Theorem 5]{LowRefocusing} the globally hyperbolic refocusing space-time $X$ has a compact Cauchy surface $M$.  Therefore, the sequence $\{\hat{y}_n\}_{n=1}^\infty$ has a convergent subsequence $\hat{y}_{n_k}\rightarrow \hat{y}\in M$.  However, this does not guarantee that $\{(\hat{y}_n,t_n)\}_{n=1}^\infty$ has a convergent subsequence in $M\times \mathbb{R}$ because it is not clear that the sequence $\{t_n\}_{n=1}^\infty$ in $\mathbb{R}$ is bounded.

Suppose, for now, that the sequence $\{t_n\}_{n=1}^\infty$ is bounded, so that there exists $C\in \mathbb{R}$ such that $|t_n|\leq C$ for all $n$.  Then $M\times [-C,C]$ is compact as a product of compact sets.  Therefore, $M\times [-C,C]\setminus V$ is compact as well since it is a closed subset of a compact set.  In this case, the sequence $\{(\hat{y}_n,t_n)\}_{n=1}^\infty$ in $M\times [-C,C]\setminus V$ has a convergent subsequence $(\hat{y}_{n_k},t_{n_k})\rightarrow (\hat{y},s)\in M\times [-C,C]\setminus V$.

Note that $(\hat{y},s)\neq (\hat{x},t)$ because $(\hat{y},s)\notin V$.  We claim that all null-geodesics through $(\hat{y},s)$ pass through $(\hat{x},t)$.  Without loss of generality we may assume all points of $(\hat{y}_{n_k},t_{n_k})$ lie in a geodesically convex neighborhood of $(\hat{y},s)$.  Let $\mathbf{v}\in T_{(\hat{y},s)}X$ be a null vector.  Parallel transport $\mathbf{v}$ to all points  $(\hat{y}_{n_k},t_{n_k})$ along the unique geodesic connecting $(\hat{y},s)$ to $(\hat{y}_{n_k},t_{n_k})$ to obtain a sequence of null vectors $\{\mathbf{v_{n_k}}\}_{k=1}^\infty$ in $T_{(\hat{y}_{n_k},t_{n_k})}X$.  Then the sequence $\{((\hat{y}_{n_k},t_{n_k}),\mathbf{v_{n_k}})\}_{n=1}^\infty$ converges to $((\hat{y},s),\mathbf{v})$ in $TX$.  Consider the map $\phi: FNX\to M$ that assigns to each future-directed null vector $\mathbf{v}$ the unique point of intersection of the null-geodesic $\gamma$  with velocity vector $\mathbf{v}$ with the Cauchy surface $M\times \{t\}$.  The continuity of $\phi$ follows from that of the exponential map. Now $((\hat{y}_{n_k},t_{n_k}),\mathbf{v_{n_k}})\rightarrow ((\hat{y},s),\mathbf{v})$ in $FNX$ so $\phi((\hat{y}_{n_k},t_{n_k}),\mathbf{v_{n_k}})\rightarrow \phi((\hat{y},s),\mathbf{v})$.  In light of \cite[Lemma 7.2]{CR} Lemma~\ref{7.2CR} it follows that all the null-geodesics through $(\hat{y},s)$ pass through $(\hat{x},t)$, so that $X$ is strongly refocusing.

However it is not known to us if the sequence of times $\{t_n\}_{n=1}^\infty$ is bounded.  Therefore we do not know if each globally hyperbolic refocusing space-time is strongly refocusing.
\end{remark}

\section{Appendix: Lorentz Geometry}

 \textbf{Appendix A: Introduction to Lorentz Geometry}\\
 
 Lorentz geometry provides a natural setting for the theory of general relativity in which space-time is represented by a Lorentz manifold.  This section introduces Lorentz geometry and defines null (lightlike) geodesics, which represent light rays in space-time.  
 
A Riemannian metric is positive-definite by definition.  This assumption can be replaced by the weaker condition of non-degeneracy, yielding a larger class of metrics called semi-Riemannian metrics.

\begin{definition}
A \emph{semi-Riemannian metric} on a smooth manifold $X$ is a non-degenerate symmetric smooth $(2,0)$-tensor field $g$ on $X$.  Thus $g$ assigns to $p\in X$ a smoothly varying choice of non-degenerate symmetric bilinear form $g_p$ on the tangent space $T_pX$ of $X$ at $p$.  A \emph{semi-Riemannian manifold} $(X,g)$ is a smooth manifold $X$ equipped with a semi-Riemannian metric $g$.
\end{definition}

\begin{definition}
An \emph{isometry} of semi-Riemannian manifolds $(X,g)$ and $(Y,h)$ is a diffeomorphism $\phi: X\to Y$ which preserves the metric, i.e.~$\phi^*h=g$, where $\phi^*h$ denotes the pullback of $h$ by $\phi$.
\end{definition}

\begin{remark}
Semi-Riemannian Geometry automatically subsumes Riemannian geometry.  This is the case because a Riemannian metric is a semi-Riemannian metric of index $0$.
\end{remark}

\begin{definition}
A \emph{Lorentz metric} is a semi-Riemannian metric of index $1$.  A \emph{Lorentz manifold} is a smooth manifold equipped with a Lorentz metric.  \emph{Lorentz geometry} is the study of properties of Lorentz manifolds which are invariant under isometries.

A non-zero tangent vector $\mathbf{v}\in T_pX$ is called \emph{spacelike} (respectively \emph{lightlike}, \emph{timelike}) if $g_p(\mathbf{v},\mathbf{v})$ is positive (respectively zero, negative).  The classification of $\mathbf{v}$ into one of these three categories is called the \emph{causal character} of $\mathbf{v}$.  Note that the latter two cases are impossible in Riemannian geometry, where the metric is positive-definite.  A synonym for lightlike is \emph{null}.  We also use the term \emph{nonspacelike} to mean not spacelike, i.e.~timelike or null. A synonym for nonspacelike is \emph{causal}.

These definitions generalize to piecewise-smooth curves and submanifolds.  A piecewise-smooth curve $\gamma$ is called \emph{spacelike} (respectively, \emph{lightlike} or \emph{timelike}) if $\mathbf{\gamma'(t)}\in T_{\gamma(t)}X$ has the corresponding property for all $t$ in the domain of $\gamma$ (except at the finite number of points where $\gamma$ fails to be smooth, at these points the one sided derivatives should satisfy the desired property).  A smooth submanifold $M$ of $X$ is called \emph{spacelike} if $g|_{M}$ is positive-definite, i.e. a Riemannian metric.

The \emph{(tangent) null-cone} of $p\in X$ is defined as the subset $N_p=\{\mathbf{v}\in T_pX: \mathbf{v}$ is lightlike$\}\subset T_pX$.  $N_p$ is disconnected and consists of two connected components which are hemi-cones.
\end{definition}

The exponential map is defined for Lorentz manifolds with respect to the unique Levi-Civita connection defined in Appendix A.  Thus we also consider the null-cone of a point as the exponential map, whose image is a subset of the manifold itself.  For distinction, the exponential map or its image is sometimes called the \emph{exponentiated} null-cone of $p$, while the null-cone in $T_pX$ is called the {tangent} null-cone.  Henceforth, 'null-cone' will sometimes refer to the 'exponentiated null-cone.  The exponentiated null-cone of a point $p\in X$ is a subset of $X$ but not necessarily a submanifold of $X$ because it may have self-intersection points, singular points and it can even be dense in $X$.

\begin{example}
The standard example of a Lorentz manifold is the \emph{Minkowski space} $\mathbb{R}^{n+1}$ with global coordinates $(x_1,...,x_n,t)$ and Lorentz metric $g=dx_1^2+...+dx_n^2-dt^2$, called the \emph{Minkowski metric}.  Minkowski space of dimension $n+1$ is denoted by $\mathbb{L}^{n+1}$ to indicate that it is a Lorentz manifold.
\end{example}

\begin{definition}
A \emph{time-orientation} of a Lorentz manifold $(X,g)$ is a continuously varying choice for each $p\in X$ of hemi-cone of $N_p$.  A nonspacelike tangent vector $\mathbf{v}$ at $p$ is called \emph{future-directed} if it lies in the hemi-cone specified by the time-orientation, and \emph{past-directed} if it lies in the opposite hemi-cone of $N_p$. A piecewise-smooth nonspacelike curve is called 
\emph{future-directed}
if all of its velocity vectors are future-directed. Similarly one defines \emph{past-directed} nonspacelike curves.

A Lorentz manifold $(X,g)$ is \emph{time-orientable} if it admits a time-orientation, and \emph{time-oriented} if it is equipped with a choice of time-orientation.
\end{definition}

If a Lorentz manifold is not time-orientable, it will always have a time-orientable double cover, see \cite[p. 51]{BEE}.

The standard time-orientation of Minkowski space is given by the constant vector field $\mathcal{Y}=\mathbf{\partial/\partial t}=(0,...,0,1)$.  With this time-orientation $(0,...,0,1)$ is future-directed and $(0,...,0,-1)$ is past-directed.

\begin{definition}
A \emph{space-time} is a connected time-oriented Lorentz manifold.  A point in a space-time is called an \emph{event}.
\end{definition}

\begin{example}
The \emph{Einstein Cylinder} is the space-time $(\mathbb{S}^n\times \mathbb{R}, g=m\oplus -dt^2)$ where $m$ is the standard (round) metric on the sphere.  The time-orientation is given by $Y=\mathbf{\partial/\partial t}$.
\end{example}

\begin{remark}
To avoid trivial cases we generally assume that the dimension of a space-time is at least two.  Without this assumption, the simplest example of a space-time is one-dimensional Minkowski space-time ($n=0$), namely $(\mathbb{R},-dt^2)$.  This is just the real line with a dot product defined by the negative of multiplication.

The Fundamental Lemma of (Semi-)Riemannian geometry states that every Lorentz manifold has a canonical Levi-Civita connection.  Curvature, geodesics and parallel transport in a space-time are defined with respect to the Levi-Civita connection of the Lorentz metric.

A nonspacelike smooth curve $\gamma$ in $X$ must be future-directed at all points or past-directed at all points.  For a geodesic curve $\gamma$ in $X$, $g_{\gamma(t)}(\mathbf{\gamma'(t)},\mathbf{\gamma'(t)})$ must be constant.  Therefore a geodesic which is spacelike (respectively timelike, null) at a point must be so everywhere.
\end{remark}

\textbf{Appendix B: Causal Structure of space-time}

\begin{definition}
Define the relations $<$ and $\leq$ on a space-time $(X,g)$ by $p< q$ if there exists a piecewise-smooth future-directed timelike curve in $X$ from $p$ to $q$ and $p\leq q$ if either $p=q$ or there exists a future-directed nonspacelike curve in $X$ from $p$ to $q$.  Clearly $p<q \Rightarrow p\leq q$ since every timelike curve is nonspacelike.

The \emph{chronological future} of $p\in X$ is $I^+(p)=\{q\in X: p<q\}\subset X$.  The \emph{chronological past} of $p$ is $I^-(p)=\{q\in X: q<p\}$.  The \emph{causal future} of $p$ is $J^+(p)=\{q\in X: p\leq q\}$ and the \emph{causal past} of $p$ is $J^-(p)=\{q\in X: q\leq p\}$  We have $q\in I^\pm(p)$ iff $p\in I^\mp(q)$.  A similar statement holds for the causal future and past.  Also $I^\pm(p)\subset J^\pm(p)$.
\end{definition}

For any event $p$, its chronological past and future $I^\pm (p)$ are open subsets of $X$.  On the other hand, the causal past and future $J^\pm (p)$ are not necessarily closed.  For example, consider the two-dimensional Minkowski space $\mathbb{R}^2$ with coordinates $(x, t)$ and metric $m=dx^2-dt^2$.  Put $(X,g)$ to be $\mathbb{R}^2\setminus (1,1)$ and $g$ to be the restriction of $m$ to $X$. Then $J^+(0,0)$ is not a closed subset of $X$.

Two events $p,q$ in a space-time are \emph{causally related} if $q\in J^+(p)$ or if $p\in J^+(q)$.  

\begin{definition}
A space-time is called \emph{chronological} if there are no closed timelike curves.  Equivalently, there does not exist an event $p\in X$ such that $p\in I^+(p).$  Similarly, a space-time is called \emph{causal} if there are no closed nonspacelike curves.  Equivalently, there do not exist distinct events $p,q\in X$ such that $p\leq q\leq p$.
\end{definition}

Every causal space-time is chronological, since timelike curves are nonspacelike, however the converse is false.

\begin{definition}
Let $(X,g)$ be a space-time.  An open subset $U\subset X$ is \emph{causally convex} if no nonspacelike curve leaving it ever returns.  In other words, no nonspacelike curve in $X$ intersects $U$ in a disconnected set.

A space-time $(X,g)$ is \emph{strongly causal at} an event $p\in X$ if $p$ has arbitrarily small causally convex neighborhoods.  $(X,g)$ is called \emph{strongly causal} if it is strongly causal at all events.
\end{definition}

\begin{remark}
 By \emph{arbitrarily small} we mean that for any open neighborhood $U$ of $p$ there exists an open neighborhood $V$ of $p$ contained in $U$ which is causally convex.  Strong causality \emph{at an event} is a local property, while strong causality is a global property of the space-time.
\end{remark}

Strong causality implies causality, while the converse is false, i.e. there are causal space-times which are not strongly causal.

\begin{definition}
The \emph{Alexandrov topology} of a space-time is defined by taking as a base all sets of the form $I^+(p)\cap I^-(q)$, $p,q\in X$.
\end{definition}

The Alexandrov topology is in general weaker than the manifold topology on any space-time because $I^+(p)\subset X$ is an open for all $p\in X$.  The following theorem asserts that strong causality is characterized by the fact that the Alexandrov and manifold topologies coincide.

\begin{theorem}\label{KP}
Kronheimer-Penrose, \cite[Theorem 4.24]{Penrose}, see also \cite[page 7]{BEE}.  For any space-time $X$ the following are equivalent:\\
\textbf{(i.)} $X$ is strongly causal.\\
\textbf{(ii.)} The Alexandrov and manifold topologies on $X$ coincide.\\
\textbf{(iii.)} The Alexandrov topology is Hausdorff.
\end{theorem}

\begin{definition}
A space-time $(X,g)$ is \emph{globally hyperbolic} if it is strongly causal and if for all $p,q\in X$, the subset $J^+(p)\cap J^-(q)\subset X$ is compact.
\end{definition}

\begin{remark}
Bernal and S\'anchez recently proved that an equivalent definition of global hyperbolicity is obtained by removing the word 'strongly' (above) since causality together with the compactness of the diamonds 
$J^+(p)\cap J^-(q)$  condition imply strong causality, see \cite[Theorem 3.2]{BernalSanchezCausal}.
\end{remark}

\begin{definition}
A \emph{Cauchy surface} in a space-time $(X,g)$ is a subset $M\subset X$ such that for any inextendible nonspacelike curve $\gamma: I\to X$, there exists a unique $t\in I$ such that $\gamma(t)\in M$.  A Cauchy surface is said to be \emph{smooth} if $M$ is a smooth submanifold of $X$.  A Cauchy surface (not assumed smooth) is also called a \emph{topological} Cauchy surface to emphasize the distinction.
\end{definition}

Globally hyperbolic space-times are characterized by the existence of a Cauchy surface, and they have a very special topological type.

\begin{theorem}
A space-time $(X^{n+1},g)$ is globally hyperbolic if and only if it admits a Cauchy surface, 
see~\cite[pages 211-212]{HawkingEllis}.  
Geroch's Theorem  is that for a globally hyperbolic $(X^{n+1},g)$ there exists a homeomorphism 
$X\cong M\times \mathbb{R}$, where $M$ is a topological $n$-manifold and each $M\times \{t\}$ is a 
topological Cauchy surface, see \cite[Theorem 3.17]{BEE}.
\end{theorem}

Bernal and S\'anchez extended Geroch's result to the smooth category in the following theorem.

\begin{theorem}
Bernal-S\'anchez, \cite[Theorem 1]{BernalSanchez}.  A space-time $(X^{n+1},g)$ is globally hyperbolic if and only if it admits a smooth spacelike Cauchy surface.  In this case, there is a diffeomorphism $X\cong M^{n}\times \mathbb{R}$ such that each $M\times \{t\}$ a smooth spacelike Cauchy surface.  Furthermore, any two smooth spacelike Cauchy surfaces in $X$ are diffeomorphic (though not necessarily isometric as Riemannian manifolds).
\end{theorem}

\begin{remark}
From now on we will assume that a Cauchy surface is smooth and spacelike unless we specify otherwise by calling it a topological Cauchy surface.  For a globally hyperbolic space-time $X\cong M\times \mathbb{R}$ we often put $M_t=M\times \{t\}$ for brevity and sometimes omit the $t$ in the subscript when the choice of $t$  is clear from the context and since each $M_t$ is diffeomorphic to $M$.  It must be kept in mind that these Cauchy surfaces are not isometric in general.
\end{remark}

\begin{definition}\label{CauchyDevelopment}
Let $(X,g)$ be a space-time and $S\subset X$ a closed subset.  The \emph{future Cauchy development} of $S$ is defined by $D^+(S)=\{x\in X:$ all past-inextendible nonspacelike curves containing $x$ intersect $S\}$.  Similarly, the \emph{past Cauchy development} of $S$ is defined by $D^-(S)=\{x\in X:$ all future-inextendible nonspacelike curves containing $x$ intersect $S\}$. The \emph{Cauchy development} of $S$ is defined by $D(S)=D^+(S)\cup D^-(S)$, the union of the future and past Cauchy developments of $S$.  Cauchy developments are also called \emph{domains of dependence} and they are covered extensively in Penrose \cite[Section 5]{Penrose}.
\end{definition}

\textbf{Appendix C: Skies and the Sky Space}

\begin{definition}
Let $(X^{n+1},g)$ be a space-time.  Let $\mathcal{N}$ be the topological space of all future-directed null-geodesics in $X$ modulo orientation-preserving affine reparametrizations.  Following \cite{LowRefocusing} we give a more precise definition of the topology on $\mathcal{N}$.  Define the \emph{future null bundle} $FNX$ of $X$ as the subspace of the tangent bundle $TX$ of all future-directed null vectors.  Define the \emph{spherical future null bundle} $SNX$ of $X$ as the quotient of $FNX$ by the group action of $\mathbb{R}^+$ on $FNX$. (Here $\mathbb{R}^+$ is the group of positive real numbers under multiplication.)  In other words, two future-directed null vectors are equivalent if and only if they differ by a positive scalar multiple.  $SNX$ can therefore be viewed as the space of all future null directions in $X$.  The space $\mathcal{N}$ of all (future-directed) null-geodesics in $X$ is defined as the quotient of $SNX$ by the following equivalence relation: two future null directions are equivalent if and only if they lie along a common null-geodesic.  A null-geodesic is viewed as the equivalence class of future null directions of its velocity vectors.
\end{definition}

The space $\mathcal{N}$ of null-geodesics in a strongly causal space-time $(X^{n+1},g)$ has a smooth $(2n-1)$-dimensional manifold structure aside from the fact that it may fail to be Hausdorff, see \cite[p. 2]{LowCones}.  For example, the punctured Minkowski space-time $(\mathbb{R}^{n+1}\setminus \mathbf{0},g=dx_1^2+...+dx_n^2-dt^2)$ is strongly causal.  Identifying each $T_p\mathbb{R}^{n+1}$ with $\mathbb{R}^{n+1}$, the sequence $\{\gamma_n\}_{n=1}^\infty$ of null-geodesics through $(0,...,0,1/n)$ in the direction $(1,0,...,0,1)$ converges in $\mathcal{N}$ to two distinct null-geodesics, namely the two connected components of the punctured line $t=x_1,  x_2=...=x_n=0$.  Since a sequence of null-geodesics in $\mathcal{N}$ converges to two different null-geodesics in $\mathcal{N}$, the space $\mathcal{N}$ fails to be Hausdorff.

\begin{definition}
Let $x\in X$ be an event.  Define the \emph{sky} of $x$ as the subspace $O_x\subset \mathcal{N}$ of all 
(future-directed) null-geodesics passing through $x$.
\end{definition}

\begin{definition}
Let $(X^{n+1},g)$ be a space-time.  We define the \emph{sky space} of $X$ as the topological space of all skies $SKY(X)=\{O_x: x\in X\}$.  The topology on this set is given by a neighborhood base at each sky as follows.

Let $W\subset \mathcal{N}$ be an open set containing $O_x$.  Assign to $W$ a neighborhood base element $B_W$ at $O_x$ to be the set of all skies $B$ contained in $W$.  Each choice of such a $W$ gives rise to a  single neighborhood base element $B_W$ at $O_x$.  Letting $W$ range over all such possibilities gives the entire neighborhood base at $O_x$, namely $B=\{B_W: W\subset \mathcal{N}$ open and containing $O_x\}$.
\end{definition}

\begin{definition}
Let $M$ be a smooth manifold and $T^*M$ its cotangent bundle.  If $M$ is a Riemannian manifold, for instance the Cauchy surface of a globally hyperbolic space-time, then its Riemannian metric induces an identification of the tangent bundle $TM$ with the cotangent bundle $T^*M$.  Let $\zeta$ be the zero-section of $T^*M$.  The \emph{spherical} (or \emph{unit}) \emph{cotangent bundle} of $M$ is defined as the quotient $ST^*M=(T^*M\setminus \zeta)/\mathbb{R}^+$ of $T^*M\setminus \zeta$ by the group action of $\mathbb{R}^+$ on $T^*M$ given by positive scalar multiplication.
\end{definition}

 Let $(X^{n+1},g)$ be a globally hyperbolic space-time and $\mathcal{N}$ the associated  space of null-geodesics in $X$ as defined above.  Then $\mathcal{N}$ is Hausdorff and has a smooth manifold structure.  Low~\cite{LowLegendrian, LowNullgeodesics} observed that in fact, $\mathcal{N}$ is diffeomorphic with the spherical cotangent bundle $ST^*M$ of any smooth spacelike Cauchy surface $M$ in $X$. We describe such a diffeomorphism below, following \cite[p. 313]{CR}.

Let $[\gamma_0]\in \mathcal{N}$ be the equivalence class of a null-geodesic $\gamma_0: I\to X$.  Maximally extend $\gamma_0$ to an inextendible null-geodesic $\gamma:J\to X$, $I\subset J$, with $\gamma(t)\in M$.  This must hold for some $t\in J$ by the definition of Cauchy surface.  Let $\mathbf{v}$ be the orthogonal projection of $\mathbf{\gamma'(t)}$ onto $T_{\gamma(t)}M$ (identified with a subspace of $T_{\gamma(t)}X$).  This defines an element $[\mathbf{v}]\in STM\cong ST^*M$ which is independent of the choice of representative of the equivalence class $[\gamma_0]\in \mathcal{N}$.  Note that $T_{\gamma(t)}X=T_{\gamma(t)}M\oplus T_{\gamma(t)}M^\perp$.  If $\mathbf{v}\in T_{\gamma(t)}M\cap T_{\gamma(t)}M^\perp$ then $\mathbf{v}=\mathbf{0}$.  (If $\mathbf{v}\neq \mathbf{0}$ then $\mathbf{v}$ is spacelike since $\mathbf{v} \in T_{\gamma(t)}M$.  On the other hand, $\mathbf{v}\in T_{\gamma(t)}M^\perp$ so $g_p(\mathbf{v},\mathbf{v})=0$ which means that $\mathbf{v}$ must be both null and spacelike, a contradiction.)

The sky $O_x$ of an event $x$ in a globally hyperbolic space-time is an embedded $(n-1)$-sphere in $\mathcal N=ST^*M$.  For $x\in M$ we have $O_x=\pi^{-1}(x)$, where $\pi: \mathcal N=ST^*M\to M$ is the projection map.  Furthermore, the spherical cotangent bundle $ST^*M$ has a canonical \emph{contact structure} (totally non-integrable smooth hyperplane distribution).  Thus for a globally hyperbolic $(X,g)$ we have that 
$\mathcal{N}=ST^*M$ is a \emph{contact manifold}. The contact structure on $\mathcal N$ does not depend on the choice of $M$, see Low~\cite{LowLegendrian} and Natario and Tod~\cite[pages 252-253]{NatarioTod}.  The sky $O_y$ of any event $y\in X$ is known to be a \emph{Legendrian submanifold} of $\mathcal N$ with respect to its canonical contact structure.  This means that the tangent space to the smooth submanifold $O_y$ of $\mathcal N$ at each point lies in the hyperplane distribution given by the contact structure on $\mathcal N$.\\

\textbf{Appendix D: The Exponential Map and Lorentz Distance}

\begin{definition}
The \emph{exponential map} of a semi-Riemannian manifold $X$ is defined by $exp: E\to X: \mathbf{v}\mapsto \gamma_\mathbf{v}(1)$, where $E$ is the open subset of $TX$ of all tangent vectors $\mathbf{v}$ such that the unique maximal geodesic $\gamma_\mathbf{v}: I \to X$ satisfying $\gamma'(0)=\mathbf{v}$ is defined on an open interval $I$ containing $[0,1]$.  For each $p\in X$, we define $E_p=E\cap T_pX$ (an open subset of $T_pX$).  We define the \emph{exponential map} of $X$ \emph{at} $p$ by $exp_p=exp|_{E_p}$.
\end{definition}

\begin{remark}
The Normal Neighborhood Lemma of Riemannian geometry holds for all semi-Riemannian manifolds.  This guarantees the existence of normal neighborhoods of each point of a semi-Riemannian manifold.  In particular, $(exp_p|_{V})^{-1}: U\to V$ is a coordinate chart on $X$ around $p$, where the coordinates in $V$ are defined with respect to a given basis of $T_pX$.
\end{remark}

\begin{definition}
An open neighborhood $U$ of a semi-Riemannian manifold $X$ is called \emph{convex} if any two points in $U$ can be joined by a unique geodesic segment lying entirely within $U$.
\end{definition}

\begin{theorem}\label{Whitehead}
Whitehead, see \cite[p. 130]{ONeill}.  Every point of a semi-Riemannian manifold has a normal neighborhood which is convex, and also has the property that it is a normal neighborhood of each point $q\in U$.
\end{theorem}

\begin{definition} A neighborhood $U$ of $p$ as above is called a \emph{convex normal neighborhood} of $p$.
\end{definition}

The following useful proposition regards the local behavior of timelike and nonspacelike curves lying entirely within a convex normal neighborhood of a point in a Lorentz manifold.

\begin{prop}\label{3.4BEE}
See \cite[Proposition 3.4]{BEE}.  Let $(X,g)$ be a Lorentz manifold, $p\in X$ and $U$ a convex normal neighborhood of $p$.  The points of $U$ which can be reached by timelike (respectively, nonspacelike) curves contained in $U$ are those of the form $exp_p(\mathbf{v})$ where $\mathbf{v}\in T_pX$ is timelike (respectively, nonspacelike).
\end{prop}

For a space-time $(X,g)$, the Lorentz metric gives rise to a function $d=d_g: X\times X\to [0,\infty]$ called the Lorentz distance function of $g$ which we define following~\cite{BEE}.  Unlike the Riemannian distance function, the Lorentz distance function is not a metric (distance function) at all, in spite of its name.  All three of the defining properties of a distance function fail to hold for the Lorentz distance function.  In fact, the Lorentz distance function is neither finite-valued nor continuous in general, however these two conditions are guaranteed by global hyperbolicity, see Theorem~\ref{ghlorentzdistance} \cite[p. 11]{BEE}.

\begin{definition}
Let $(X,g)$ be a space-time, and let $\gamma:[a,b]\to X$ be a piecewise-smooth curve in $X$.  There are numbers $a_i\in[a,b], i=0,...,k$ satisfying $a=a_0<a_1<a_k=b$, such that $\gamma$ is smooth when restricted to each $[a_i,a_{i+1}],  i=0,...,k-1$.  We define the \emph{Lorentz arc-length} of $\gamma$ by $L(\gamma)=\sum_{i=0}^{k-1}\int_{a_i}^{a_{i+1}}\sqrt{-g_{\gamma(t)}(\mathbf{\gamma'(t)},\mathbf{\gamma'(t)})}~dt$.  Lorentz arc-length is independent of the choice of parametrization of $\gamma$.
\end{definition}

Given two events $p,q\in X$ with $q\in J^+(p)$, denote by $\Omega_{p,q}$ be the set of all piecewise-smooth nonspacelike curves in $X$ from $p$ to $q$.

\begin{definition}
We define the \emph{Lorentz distance function} $d: X\times X\to [0,\infty]$ by $d(p,q)=0$ if $q\notin J^+(p)$, and $d(p,q)=\sup\{L(\gamma):\gamma\in \Omega_{p,q}\}$ if $q\in J^+(p)$.
\end{definition}

In general, the Lorentz distance function fails to be symmetric.  The Lorentz distance from one event to another, distinct event may be zero. A reverse analogue of the triangle inequality holds for causally related events: if $p\leq q \leq r$ then $d(p,r)\geq d(p,q)+d(q,r)$, see \cite[p. 9]{BEE}.  These properties illustrate the fact that the Lorentz distance function is in fact \emph{not actually a distance function}.

\begin{theorem}\label{ghlorentzdistance}
See \cite[Corollary 4.7]{BEE}.  For a globally hyperbolic space-time $(X,g)$, 
the Lorentz distance function $d_g$ on $X\times X$ is continuous and it takes values in $\mathbb{R}$ 
$($rather than in $\mathbb{R}\sqcup \infty$$)$.
\end{theorem}

\begin{remark}
As opposed to the Riemannian geometry, timelike geodesics in a Lorentz manifold locally \emph{maximize} arc-length, see for instance \cite[p. 146]{BEE}.  We conclude the discussion of Lorentz distance with the following useful definition.
\end{remark}

\begin{definition}
Let $(X,g)$ be a space-time.  A \emph{local Lorentz distance function} on $(X,g)$ is a convex normal neighborhood $U\subset X$ equipped with the Lorentz distance function $D: U\times U\to \mathbb{R}$ of the space-time $(U,g|U)$.  Therefore, for $x,y\in U$ we have $D(x,y)=0$ if  there does not exist a future-directed timelike geodesic segment in $U$ from $x$ to $y$.  Otherwise $D(x,y)$ is the Lorentz arc-length of the unique such geodesic segment in $U$.
\end{definition}

\begin{prop}
See \cite[Theorem 4.27]{BEE}.  A space-time is strongly causal if and only if each event has a convex normal neighborhood on which the local Lorentz distance function agrees with the (global) Lorentz distance function of the space-time.
\end{prop}

\end{document}